\documentclass[12pt, a4paper]{article}
\usepackage{amsmath, amsthm, amssymb, enumitem, graphicx, color, hyperref, svg}
\usepackage[square,numbers]{natbib}
\allowdisplaybreaks

\newtheorem{theorem}{Theorem}
\newtheorem*{maintheorem}{Main Theorem}

\newtheorem{defin}{Definition}
\newtheorem{lemma}{Lemma}

\newcommand{\bZ}{\mathbb{Z}}
\newcommand{\bR}{\mathbb{R}}

\newcommand{\into}{\hookrightarrow}
\newcommand{\vol}{\operatorname{vol}}
\newcommand{\fillrad}{\operatorname{FillRad}}
\newcommand{\fillvol}{\operatorname{FillVol}}
\newcommand{\cage}{\operatorname{Cage}}

\begin{document}
\author{Isabel Beach}
\title{The Shortest Geodesic Flower on a Manifold with Locally Convex Ends and Finite Volume}
\maketitle

\begin{abstract}
    Suppose $M$ is a complete, non-compact $n$-dimensional Riemannian manifold with locally convex ends and finite volume. We prove that $M$ admits a non-trivial geodesic net with one vertex, at most $(n+2)(n+1)/2$ edges, and total length at most $$(n+2)(n+1)(n/2)\vol(M)^{1/n}.$$ This is a quantitative version of the main result of G. R. Chambers, Y. Liokumovich, A. Nabutovsky and R. Rotman in \cite{clnr}.
\end{abstract}

\section{Introduction}

Geodesic nets are the critical points of the length/energy functional on the space of immersed multigraphs in a given Riemannian manifold without boundary (see Definition \ref{def:net}). 
In this article, we prove the following quantitative result.
\begin{maintheorem}
    \label{theorem:main}
    Let $M$ be a complete, non-compact $n$-dimensional Riemannian manifold with locally convex ends and finite volume. Then $M$ admits a non-trivial geodesic net with one vertex, at most $(n+2)(n+1)/2$ edges, and total length at most $(n+2)(n+1)(n/2)\vol(M)^{1/n}.$
\end{maintheorem} 
\noindent 
The notion of a locally convex end is given in Definition \ref{def:ends}. Essentially, this means that $M$ can be divided into a bounded region and some number of unbounded regions that are locally geodesically convex.
\par
One motivation for the study of geodesic nets is the fact that they arise as minimal varifolds.
The Almgren--Pitts min-max values for $p$-sweepouts through 1-cycles on a given manifold, denoted $\omega_p^1$, are realized by the lengths of geodesic nets on that manifold. This means that for each $p$, there is a geodesic net on the associated manifold whose length equals $\omega_p^1$. In fact, it was proven by O. Chodosh and C. Mantoulidis in 2023 that, on surfaces, these widths are realized by the lengths of closed geodesics (i.e., geodesic nets with one edge and one vertex) \cite{chodosh_mantoulidis}.
M. Gromov conjectured in 2003 that these widths should satisfy a Weyl Law, so that on a compact $n$-dimensional manifold $M$ we would have
$$
\lim_{p\to\infty}
\frac{\omega_p^1}{p^{(n-1)/n}}=a(n)\vol(M)^{1/n}
$$
for some constant $a(n)$ \cite{gromov_waists}.
This was recently proven in dimension three by L. Guth and Y. Liokumovich \cite{lio_guth}, and later in all dimensions by B. Staffa \cite{staffa2025weyllaw1cycles}. 
Staffa also proved a version of the bumpy metrics theorem for geodesic nets in 2023 \cite{staffa2023nets}, meaning that a generic metric admits only embedded geodesic nets that are non-degenerate (i.e., that do not have non-parallel Jacobi fields). This was also proven independently by O. Chodosh and C. Mantoulidis in 2023 \cite{chodosh_mantoulidis}.
\par
Studying geodesic nets can also be a first step in understanding closed geodesics. We are interested in using geodesic nets to address the following two questions regarding closed geodesics. The first question we consider, posed by V. Bangert, is to determine which non-complete Riemannian manifolds with finite volume admit closed geodesics. Every surface with finite area admits a closed geodesic by G. Thorbergsson \cite{thorbergsson1978} and Bangert \cite{bangert1980}, but this question is open in higher dimensions. In 2023, G. R. Chambers, Y. Liokumovich, A. Nabutovsky and R. Rotman conjectured that every complete, non-compact manifold with locally convex ends admits a closed geodesic \cite{clnr}. Under certain topological conditions, this conjecture was proven by A. Dey in \cite{dey2024}. There are also additional existence results in specific cases due to V. Benci and F. Giannoni \cite{bencigiannoni1991} and L. Asselle and M. Mazzucchelli \cite{assellemazzucchelli2016}.
\par  
The second question we consider, due to Gromov, is whether the length of a shortest closed geodesic on a compact $n$-dimensional manifold with finite volume is bounded by $c(n)\vol(M)^{1/n}$, where $c(n)$ is some universal constant that depends only on $n$. 
There are many results concerning length bounds for a shortest closed geodesic on a closed surface (for a survey, see e.g. \cite{CrokeKatz2003}). There are also results for non-compact surfaces. For example, C. B. Croke proved that a shortest closed geodesic on a surface $M$ with finite area has length at most $31\sqrt{ \vol(M)}$ \cite{croke1988}. This bound was improved by the author and Rotman in 2020 to at most $4\sqrt{2 \vol(M)}$ \cite{my_baby}. However, there are few results in dimensions greater than two (see, for example, \cite{rotman2000, nabutovsky2002}). 
\par
In light of the above, we are interested in the following two questions concerning geodesic nets. First, which non-compact Riemannian manifolds with finite volume admit geodesic nets? Second, how long is a shortest geodesic net on a non-compact Riemannian manifold with finite volume?
In 2023, Chambers, Liokumovich, Nabutovsky and Rotman proved that there exists a geodesic net on every complete, non-compact manifold with locally convex ends (but not necessarily with finite volume) \cite{clnr}, giving a partial answer to the first question. Moreover, the net they obtain has at most one vertex.
There are currently no published answers to the second question, although there are several results concerning length bounds for geodesic nets on closed manifolds, such as \cite{rotman2005, rotman2007, rotman_flowers}. For example, in \cite{rotman_flowers}, Rotman proved that 
a closed $n$-dimensional Riemannian manifold admits a geodesic net with at most $3^{(n+1)^2}$ edges, one vertex, and total length at most 
$$		2((n + 1)!)^{5/2} 3^{(n+1)^3} (n + 1)n^n \vol(M)^{1/n}.$$
In the present article, we build upon \cite{clnr} to prove a length bound for a shortest geodesic net on a complete, non-compact manifold with locally convex ends and finite volume. This partially answers the second question posed above. Our result differs from previous quantitative results for geodesic nets because it concerns a manifold that is not compact. It also differs from the quantitative results for closed geodesics on surfaces as it applies in dimensions greater than two.

\section{Results}

Our proof of our main theorem builds upon the existence proof of \cite{clnr}. However, unlike \cite{clnr}, we want to ensure that the geodesic net we obtain has bounded length. The proof uses Gromov's technique of pseudo-extension. Essentially, we will take a map $\phi$ and prove that a certain extension $\psi$ of $\phi$ to a larger domain cannot exist for topological reasons. We will then prove that under certain geometric assumptions (in our case, the absence of a short geodesic net with one vertex) we can in fact construct such an extension. The resulting contradiction proves that our geometric assumption cannot hold.
\par 
This article is organized as follows.
In Section \ref{sec:convex}, we define what it means for the manifold $M$ to have locally convex ends. Using the fact that $M$ has finite volume, we then use the co-area formula to cut off the infinite ends of $M$ by hypersurfaces of arbitrarily small volume. Our initial map, $\phi$, will be (up to a small homotopy) the inclusion of the bounded subset of $M$ we obtain after cutting off its infinite ends into a compactified version of $M$. In Section \ref{sec:nets}, we provide the required background material on geodesic nets. In particular, we recall from \cite{clnr} that if there are no ``short" geodesic nets on $M$, then a continuous function on the 1-skeleton of a triangulated space can be continuously extended to a map on the entire space. In this context, a net is ``short" if it is no longer than the image under $\psi$ of the 1-skeleton of any simplex. In Section \ref{sec:fillings}, we provide the required background on fillings by chains and the associated geometric inequalities. In Section \ref{sec:extension}, we construct our extended domain, denoted $F(\Gamma)$, by filling each hypersurface of $M$ we previously cut along by a homological chain, and then filling the resulting space by another chain. We use topological methods to prove in Lemma \ref{lemma:no_extension} that $\phi$ cannot have its domain continuously extended to $F(\Gamma)$. This lemma is the backbone of our pseudo-extension proof. Finally, in Section \ref{sec:main}, we prove our main theorem. We assume that there are no ``short" geodesic nets on $M$ and show how to construct the pseudo-extension map $\psi$ using the results of Section \ref{sec:nets}. The resulting contradiction proves that there is in fact a geodesic net with bounded length, proving our claim. 
\par 
The geodesic net we obtain has length bounded by the maximal length of the image under $\psi$ of the 1-skeleton of any simplex in our triangulation of $F(\Gamma)$. Edges will be mapped under $\psi$ to minimizing geodesics, so it will suffice to bound the distance in $M$ between the images of neighbouring vertices in the triangulation. By the triangle inequality, we can bound this distance by bounding the distance between any vertex in $F(\Gamma)$ and its image. Our claimed length bound then results from bounding the geometry of the extended domain $F(\Gamma)$. Since this space is defined using fillings by chains, we will make use of filling inequalities (see Section \ref{sec:fillings}). The exact constant in our bound arises from the constants in these filling inequalities.

\subsection{Locally Convex Ends}
\label{sec:convex}

Our setting is a complete, non-compact Riemannian manifold $M$ with finite volume. We also require that $M$ has locally convex ends, as defined below. 
\begin{defin}[c.f. Definition 1.1 in \cite{clnr}]
    \label{def:ends}
    A complete, non-compact Riemannian manifold $M$ is said to have locally convex ends if for some $p>1$ there exist open subsets $M^\Sigma_i\subset M$, $0\leq i\leq p$, with the following properties.
    \begin{enumerate}
        \item 
        $M\setminus \cup_{i=0}^p M^\Sigma_i$ is a disjoint union of closed hypersurfaces $\Sigma_1,\cdots \Sigma_p$ such that 
        \begin{enumerate}
            \item $M^\Sigma_0$ is bounded and $\partial M^\Sigma_0=\cup_{i=1}^p \Sigma_i$.
            \item For $i\geq 1$, $M^\Sigma_i$ is unbounded and $\partial M^\Sigma_i=\Sigma_i$.
        \end{enumerate}
        \item 
        There exists $\delta>0$ such that for $i\geq 1$, if $x,y\in \overline{M^\Sigma_i}$ with $d(x,y)\leq \delta$, then a minimizing geodesic connecting $x$ and $y$ lies in $\overline{M^\Sigma_i}$.
    \end{enumerate}
\end{defin} 
\noindent The convexity condition on the ends of $M$ allows us to control the behaviour of objects under a length shortening flow. In particular, when we apply a length shortening flow to an immersed multigraph in $M$ (see Section \ref{sec:fillings}), if this graph ever enters one of the locally convex infinite ends, it will remain there for all time.
\par
In our proof, we will utilize an additional method of ``cutting off" the infinite ends of $M$. The hypersurfaces $\Sigma_i$ just described allow us to cut $M$ into a bounded core and some unbounded locally convex regions. However, we have little control over the geometry of the hypersurfaces $\Sigma_i$ or the ends $M^\Sigma_i$. Therefore we will also divide $M$ up by hypersurfaces with arbitrarily small volume, which we will denote by $\Gamma_i$. These hypersurfaces split $M$ into a bounded core and some unbounded regions whose boundaries have arbitrarily small volume, giving us some control on the geometry of the infinite ends.

\begin{lemma}
    \label{lemma:gamma_ends}
    Let $M$ be a complete, non-compact Riemannian manifold $M$ with locally convex ends $M_i^\Sigma$, $1\leq i\leq p$, and finite volume. Then for any $\epsilon>0$, there exists $m>0$ and subsets $\Gamma_i\subset M$, $1\leq i \leq m $, satisfying the following.
    \begin{enumerate}
    \item 
        Each $\Gamma_i$ is a disjoint union of closed hypersurfaces.
        \item 
        Each $\Gamma_i$ has $(n-1)$-volume at most $\epsilon$.
        \item 
        $\cup_{i=1}^m\Gamma_i$ splits $M$ into a disjoint union of submanifolds $M^\Gamma_0,\cdots ,M^\Gamma_m$ such that 
        \begin{enumerate}
            \item $M^\Gamma_0$ is bounded and $\partial M^\Gamma_0=\cup_{i=1}^m \Gamma_i$.
            \item For $i\geq 1$, $M^\Gamma_i$ is unbounded and $\partial M^\Gamma_i=\Gamma_i$.
            \item 
            The bounded core $M^\Gamma_0$ contains the bounded core $M^\Sigma_0$ derived by cutting off the locally convex ends $M_i^\Sigma$.           
            \item 
            Each infinite end $M^\Gamma_i$ is contained by one of the locally convex infinite ends $M_j^\Sigma$.
        \end{enumerate}
        \end{enumerate}
\end{lemma}
\begin{proof}
    We would like to cut off the infinite ends of $M$ by a very large metric sphere. However, we need to ensure that this metric sphere consists of closed submanifolds with small volume. Fix a point $p\in M$.
    By the work of M. P. Gaffney in \cite{gaffney}, for any $\epsilon>0$ there exists a smooth function $\tilde{d}$ that is within $\epsilon$ of the distance function $d(p,\cdot)$ and whose gradient has norm at most $1+\epsilon$. Let $\tilde{S_r}$ be the level set of points satisfying $\tilde{d}(p,\cdot)=r$. 
    Because $M$ has finite volume, by the co-area formula (see e.g. \cite{burago_zalgaller_book} \S 13.4) we have 
    $$
    \infty >(1+\epsilon)\vol(M)\geq \int_M |\operatorname{grad}\tilde{d}(x)|dx =\int_0^\infty \vol(\tilde{S_r})dr.
    $$ 
    Therefore for any $R>0$ and $\delta>0$, the set of values $r>R$ so that $\tilde{S_r}$ has volume at most $\delta$ has positive measure. By Sard's Theorem, $r$ is a regular value of $\tilde{d}$ for all but a measure zero set of values. 
    Because $M_0^\Sigma$ is bounded, there is some $T>0$ so that $\tilde{S_T}$ entirely contains $M_0^\Sigma$. Consider some choice of regular value $r>T$ such that $\tilde{S_r}$ has volume at most $\delta$. Then $\tilde{S_r}$ bounds a bounded region $M^\Gamma_0$ such that $M^\Gamma_0\supseteq M^\Sigma_0$. Define the connected components of $M\setminus \overline{M^\Gamma_0}$ as $M^\Gamma_i$, $1\leq i \leq m$. If any $M_i^\Gamma$ is compact, we redefine $M_0^\Gamma$ as $M_0^\Gamma\cup M_i^\Gamma$ and forget about that particular $M^\Gamma_i$. Then each $M^\Gamma_i$ lies in some $M^\Sigma_j$. We then define $\Gamma_i=\partial M_i^\Gamma\subseteq \tilde{S_r}$, noting that these are (possibly disconnected) hypersurfaces without boundary that have volume at most $\epsilon$. 
    \par
    Lastly, we want to check that the $\Gamma_i$ are closed. Note that $\tilde{S_{r}}$ is the preimage of the (topologically) closed set $\{r\}\subset \bR$, and therefore is (topologically) closed because $\tilde{d}$ is continuous. Moreover, it is contained in the metric ball of radius $r+\epsilon$. Therefore it is (topologically) closed and bounded, and hence is compact by the Hopf-Rinow theorem. Therefore $\tilde{S_r}$ (and hence each $\Gamma_i$) consists of a disjoint union of closed (i.e., compact and without boundary) hypersurfaces.
\end{proof}

\subsection{Geodesic Nets and Cages}
\label{sec:nets}

We now discuss the required background material concerning geodesic nets.
\begin{defin}
    \label{def:net}
    A net is an immersed multigraph in $M$. A geodesic net is a net whose edges are geodesics and such that the unit tangent vectors of every edge incident to a given vertex add up to zero.
    \par 
    A flower is a net with only one vertex.
    A geodesic flower is a geodesic net with only one vertex.
\end{defin}
\noindent
Note that some texts refer to the above objects as \textit{stationary} or \textit{minimal} geodesic nets and do not require geodesic nets to satisfy any condition on their vertices. 
As previously noted, geodesic nets are also the critical points of the length functional over the space of multigraphs.
\par
In order to define our maps $\phi$ and $\psi$, we need the notion of a \textit{weak filling of cages} on a manifold. If a weak filling of cages exists on $M$, then a map from a simplicial complex into $M$ that is only defined only on the 1-skeleton of the complex can be continuously extended to the entire complex. In particular, we will use this technique to define $\phi$ and $\psi$.
To define a weak filling of cages, we need to define a compactified version of $M$,
since the weak filling of cages will be defined on cages in the compactified space $C^\Sigma$ instead of $M$ in order to allow cages to converge to a point ``at infinity". Consider a locally convex infinite end $M_i^\Sigma$ given by cutting $M$ along $\Sigma_i$. Let $C_i^\Sigma$ indicate the one-point compactification of $\overline{M_i^\Sigma}$ (here, $C$ indicates compactification). After compactifying every end of $M$, we obtain a space $M_0^\Sigma\cup_i C_i^\Sigma$ which we denote by $C^\Sigma$. 
We also define the quotient map
$$q:(C^\Sigma,\cup_{i=1}^m C^\Sigma_i)\to (C^\Sigma/\cup_{i=1}^m C^\Sigma_i,*).$$
We are now ready to define a weak filling of cages.
\begin{defin}
    \label{def:cages}
    An $i$-cage is the image of the 1-skeleton of an $i$-simplex in $M$ such that the image of each edge is piecewise geodesic. The set of $k$-cages in $M$ with total length at most $L$ is denoted $\cage_k^L$.
\end{defin}
\begin{defin}[c.f. Definition 3.1 and Definition 3.3 in \cite{clnr}]
    \label{def:filling_cages} A weak filling of $\cage_k^L$ is a collection of (possibly discontinuous) maps 
    $$H_i:\cage^L_i\times[0,1]\to \cage^L_i$$ 
    for $2\leq i\leq k$ with the following properties.
    \begin{enumerate}
        \item 
        $H_i(\cdot,0)$ is the identity map and $H_i(\cdot,1)$ maps every cage to a constant cage.
        \item 
        $H_i(\cdot,t)$ fixes the set of constant cages pointwise for every $t\in[0,1]$.
        \item 
        If an $i$-cage lies in a locally convex subset, this remains true after applying $H_i(\cdot,t)$ for any $t\in[0,1)$.
        \item 
        The image under $H_i(\cdot,t)$ of a flower with at most $j$ edges is also a flower with at most $j$ edges for any $t\in[0,1]$. 
        \item 
        We can define a map $\tilde{q}$ from the space of cages in $M$ to the space of cages in $ C^\Sigma/\cup_{i=1}^m C^\Sigma_i$ by applying $q$ to the image of each cage. Then each map $\tilde{q}\circ H_i(\cdot, t)$ is continuous for all $t\in[0,1]$.\qed
    \end{enumerate}
\end{defin}

Weak fillings have the following crucial properties which allow us to use them to extend a map defined only on the 1-skeleton of a simplicial complex.

\begin{lemma}[Proposition 3.4 in \cite{clnr}]
    \label{lemma:filling}
    Suppose $\{H_i\}_{i=2}^k$ is a weak filling of $\cage_k^L$. Let $\Theta\in \cage_i^L$, $i\leq k$. Then the following properties hold.
    \begin{enumerate}
        \item 
        There is a map $g_\Theta$ of the $i$-simplex into $C^\Sigma$ that extends $\Theta$ such that $q\circ g_\Theta$ is a continuous map that depends continuously on $\Theta$.
        \item 
        If $\Theta'$ is a subcage of $\Theta$ representing a face of the $i$-simplex, then $g_{\Theta'}$ equals the restriction of $g_\Theta$ to that face.
        \item 
        If $\Theta$ lies in a locally convex subset, then so does the image of $g_\Theta$. \qed
    \end{enumerate}
\end{lemma}
\noindent
Roughly speaking, we extend an $i$-cage $\Theta$ to the associated $i$-simplex by induction on $i$. As base case, suppose that $\Theta$ is a 2-cage bounding a 2-simplex. By applying the weak filling map to $\Theta$, we contract this cage to a point. Since the 2-simplex is a disk, we can view it as a 1-parameter family of 2-cages $\Theta_t$ for $t\in[0,1]$, with $\Theta_0=\Theta$ and $\Theta_1$ a point. Then we can define the image of $\Theta_t$ as $H_2(\Theta,t)$, defining our extension map. If $\Theta$ is an $i$-cage for $i>2$, suppose that we are able to extend all cages of dimension at most $i-1$. Then at each time $t\in[0,1]$, we can extend each subcage of $H_i(\Theta,t)$ corresponding to an $(i-1)$-face of the $i$-simplex associated to $\Theta$. This gives us a 1-parameter family of $(i-1)$-dimensional spheres that fills the $i$-simplex, defining our map.
\par
In the interest of being able to apply the above lemma to define our pseudo-extension map $\psi$, we must first ensure that a weak filling exists. This is the purpose of the following theorem.

\begin{theorem}[Lemma 3.10 and Theorem 3.11 in \cite{clnr}]
\label{theorem:filling_no_cages}
    Suppose $M$ is a complete non-compact $n$-dimensional Riemannian manifold with locally convex ends. 
    Given any $2\leq k \leq n+1$, if $M$ does not admit a non-trivial geodesic flower with at most ${ k+1\choose 2}$ edges and total length at most $L$, then there exists a weak filling of $\cage_k^L$. \qed
\end{theorem}
\noindent Note that ${ k+1\choose 2}$ is the number of edges of a $k$-cage.
This result holds because it is possible to construct a weak filling using a ``net shortening flow" per the following argument from \cite{clnr}. First, note that $\cage^L_k$ can be flowed to the subset of flowers in $\cage^L_k$ by the following process. For each net, fix a ``base" vertex and continuously move each other vertex to this base vertex along the edge connecting the two vertices. The edges are then redefined continuously so that the cage has the same image at each point in time. Therefore it is enough to define the filling maps $H_i$ on all flowers in the space $\cage_i^L$. 
\par
\begin{figure}
    \centering
    \includegraphics[width=0.7\textwidth]{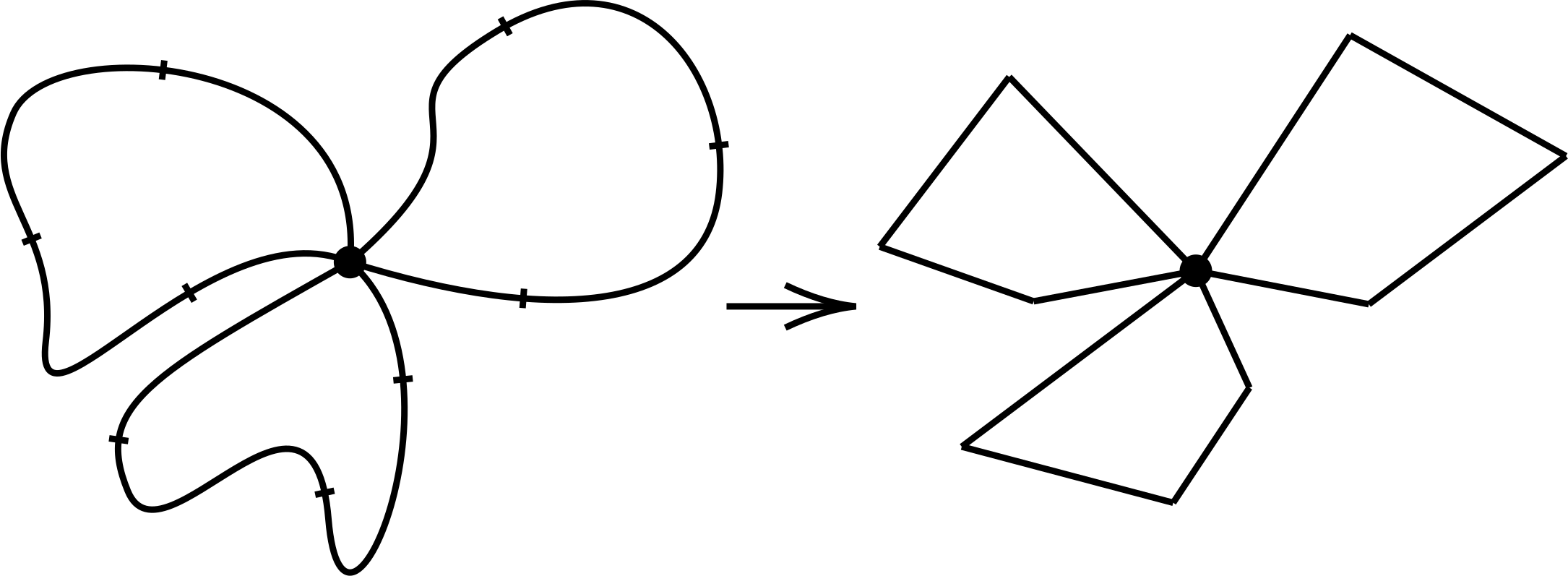}
    \caption{Shortening the edges of a flower using Birkhoff's process.}
    \label{fig:birkhoff}
\end{figure}
As a first step in defining $H_i$, we define a length shortening flow on a flower by fixing its central vertex and applying the Birkhoff curve shortening process to each edge individually. This amounts to picking a sequence of $k$ evenly spaced points along each edge, with the endpoints of the edge as the first and last points in the sequence, and then connecting each adjacent pair of points by a minimizing geodesic (see Figure \ref{fig:birkhoff}). The value of $k$ should be chosen large enough that each pair of consecutive points is connected by a unique minimizing geodesic. Under this process, every flower converges to a flower whose edges are geodesics. However, the stationarity condition at the vertex will not necessarily be satisfied, so these flowers are not necessarily geodesic flowers. Therefore, we then move the flower's vertex a small distance in the direction of the sum of the outward-pointing unit tangent vectors of each edge incident to the vertex (see Figure \ref{fig:tangent}). Note that this step does not increase the flower's total length. After repeating these two steps continually, the flower will converge to either a point or a geodesic flower.

\begin{figure}
    \centering
    \includegraphics[width=\textwidth]{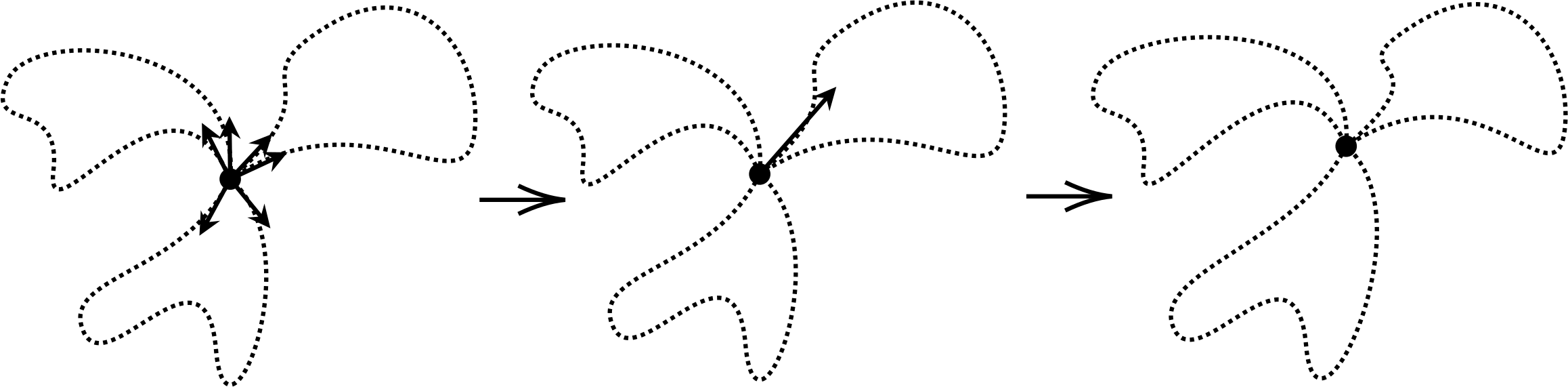}
    \caption{Ensuring that the stationarity condition holds.}
    \label{fig:tangent}
\end{figure}
\begin{figure}
    \centering
    \includegraphics[width=0.5\textwidth]{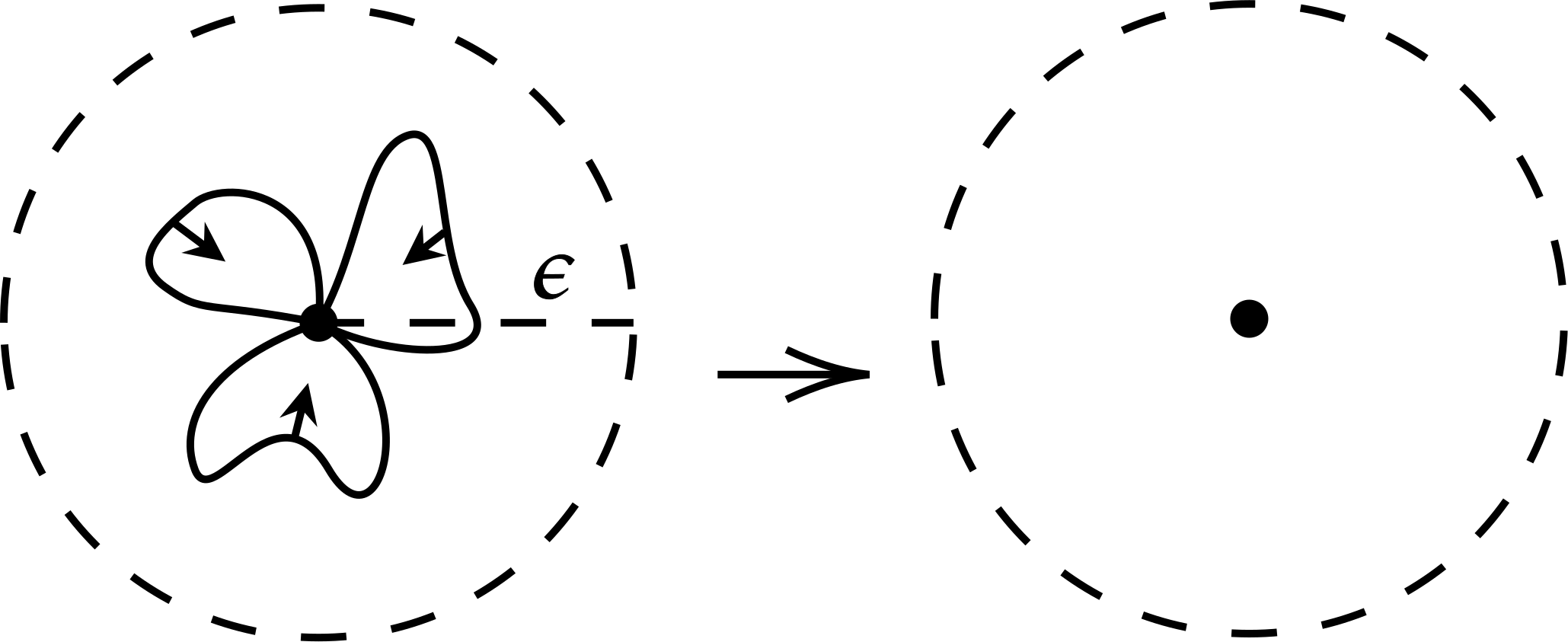}
    \caption{Contracting a small net to a point.}
    \label{fig:tiny_net}
\end{figure}

\par 
We build upon this procedure to define our weak filling on flowers in $\cage^L_{k}$ as follows. Flowers that lie entirely outside of a $2L$-neighbourhood of the bounded core $M_0^\Sigma$, and hence entirely inside one of the locally convex ends $M_i^\Sigma$, will be mapped by $H_k(\cdot, t)$ to the point at infinity in $C_i^\Sigma$ for $t\in(0,1]$. Consider instead a flower that intersects the $2L$-neighbourhood of $M_0^\Sigma$. Let $r_{inj}$ be the injectivity radius of this neighbourhood. A flower whose image lies within $r_{inj}$ of its vertex can be flowed to a point by applying the standard Birkhoff process to its edges (see Figure \ref{fig:tiny_net}). This is because no $r_{inj}$-ball in this subset completely contains a geodesic flower. If the flower does not lie within $r_{inj}$ of its vertex, then we apply our above shortening procedure to the flower until it either converges to a point curve or a geodesic flower. If we assume that the $2L$-neighbourhood of $M_0^\Sigma$ does not admit a geodesic flower of length at most $L$, then all flowers in this region must converge to point curves and will do so in finite time. We can therefore scale time so that all such flowers converge in unit time. We then define the weak filling map $H_k(\cdot,t)$ on a flower in $\cage_{k}^L$ by flowing it for time $t$ as described above.
Note that the maps $H_k$ are not continuous themselves. This is because all flowers that lie outside the $2L$-neighbourhood of $M_0^\Sigma$ are mapped by $H_k(\cdot,t)$ to a point at infinity for all $t>0$. However, this discontinuity is resolved by applying $q$.

\subsection{Filling by Chains\texorpdfstring{ in $L^\infty(M)$}{}}
\label{sec:fillings}

We would like to (attempt to) extend the inclusion map $\phi: \overline{M^\Gamma_0}\into C^\Sigma$ to a larger domain. However, in our for our pseudo-extension proof to work we need to find a domain it is \textit{impossible} to extend $\phi$ to. Our inspiration is the identity map $S^n$, whose domain cannot be extended continuously to the solid ball $B^n$. Following this analogy, we want our extended domain to be an $(n+1)$-dimensional space whose boundary is, roughly speaking, $\overline{M^\Gamma_0}$. Of course, $\overline{M^\Gamma_0}$ is not a closed manifold, so we will first fill each boundary component by a compact space. 
In order to accomplish this construction, we need the following definitions per Gromov in \cite{gromov_filling}.
\par
Consider a Riemannian manifold $N$. It is possible to embed $N$ inside $L^\infty(N)$, the space of bounded Borel functions on $N$. This is done via the Kuratowski embedding, which maps every point $x\in N$ to the function $d_x:y\mapsto d(x,y)$. The space $L^\infty(N)$ comes equipped with the sup norm $d_\infty$. This embedding is isometric in the sense that $d(x,y)=d_\infty(d_x,d_y)$ for any $x,y\in N$. Viewing $N$ in this larger space, we can consider singular $(n+1)$-chains that fill it-- that is, whose boundary represents the fundamental class of $N$. This is how we will construct the domain of our extended map. 
\par 
We make the following geometric definitions.
\begin{defin}
    \label{def:fillrad}
    Let $z$ be an $n$-dimensional cycle in $L^\infty(N)$. We define the filling radius of $z$ by
    \begin{align*}
        \fillrad(z) = \inf_{\epsilon>0} \{ z\text{ is a boundary in its }\epsilon\text{-neighbourhood in } L^\infty(N).\}
    \end{align*}
\end{defin}

\begin{defin}
    \label{def:fillvol}
    Let $z$ be an $n$-dimensional cycle in $L^\infty(N)$. We define the filling volume of $z$ by
    \begin{align*}
        \fillvol(z) = \inf_{\partial w = z} \vol w,
    \end{align*}
    where the volume of a simplex $\sigma:\Delta^{k}\to L^\infty(N)$ in a chain is defined as the smallest possible volume of $\Delta^{k}$ with respect to a metric on $\Delta^{k}$ that makes $\sigma$ distance-decreasing.
\end{defin}
These two quantities obey the following inequalities. 
Gromov's Theorem 4.2.B of \cite{gromov_filling} states that any piecewise smooth 
$m$-dimensional cycle $z$ in $L^\infty(N)$ satisfies
\begin{align}
\label{eqn:fillvol}
    \fillvol(z)
    < m!\sqrt{(m+1)!}\vol_{m}(z)^{(m+1)/m}.
\end{align}
Second, Theorem 4.3.B. in \cite{gromov_filling} states that any
$m$-dimensional cycle $z$ in $L^\infty(N)$ satisfies
$$
\fillrad(z) < (m+1)m^m\sqrt{m!}\vol(z)^{1/m}.
$$
Gromov's work was built upon by several authors, such as Guth \cite{guth_torus, guth_balls, guth_urysohn}, P. Papasoglu \cite{papasoglu}, and S. Wegner \cite{wegner}.
In Theorem 1.2 of Nabutovsky's work \cite{nabutovsky2022}, the filling radius inequality is improved for any closed Riemannian manifold $N$ of dimension $n$ to
\begin{align}
\label{eqn:fillrad}
    \fillrad(N)\leq \frac{n}{2}\vol_{n}(N)^{1/{n}},
\end{align}
where $\fillrad(N)$ is to be understood as the smallest possible filling radius of a cycle realizing the fundamental class of $N$. In fact, the proof of this inequality shows that it is also true for $n$-dimensional cycles in $L^\infty(N)$.
In the following section, we will utilize these inequalities to control the geometry of our extended domain.

\subsection{Pseudo-Extension Lemma}
\label{sec:extension}

We now have the required background to outline our pseudo-extension proof. Our initial domain will be the closure of the bounded core we obtain after cutting off the infinite ends of $M$ along the small-volume hypersurfaces $\Gamma_i$. 
Our initial map $\phi$ is (up to a small perturbation which we describe in the following section) the inclusion map of $\overline{M_0^\Gamma}$ into $C^\Sigma$.
To define the domain of our extended map, we would like to fill each $\Gamma_i$ by a chain in $L^\infty(M)$ with small volume. This will ensure that the space constructed by gluing these fillings to $M_0^\Gamma$ has small volume and hence a small filling radius by Inequality \ref{eqn:fillrad}, which will ultimately allow us to bound the total length of the geodesic net we obtain in Section \ref{sec:main}. By Inequality \ref{eqn:fillrad}, for each $\Gamma_i$ we have 
    \begin{align*}
        \fillrad(\Gamma_i)\leq \frac{n-1}{2}\vol_{n-1}(\Gamma_i)^{1/{n-1}}.
    \end{align*}
By Inequality \ref{eqn:fillvol}, we have
\begin{align*}
    \fillvol(\Gamma_i)
    < (n-1)!\sqrt{n!}\vol_{n-1}(\Gamma_i)^{n/(n-1)}.
\end{align*}
Recall that we can choose the $\Gamma_i$ so that $\vol_{n-1}(\Gamma_i)$ is arbitrarily small.
Therefore, given any $\epsilon>0$ we can choose $\Gamma_i$ so that both $\fillrad(\Gamma_i)$ and $\fillvol(\Gamma_i)$ are smaller than $\epsilon$. The same is also true of each connected component of $\Gamma_i$. Therefore it is possible to find a singular chain $F(\Gamma_i)$ in $L^\infty(M)$ that is a disjoint union of chains that fill each connected component of $\Gamma_i$, that has volume at most $\epsilon$, and that lies within an $\epsilon$-neighbourhood of $\Gamma_i$. Note that we have already established that each component of $\Gamma_i$ is a closed submanifold, and hence can be filled.
\par 
Next, we view the fillings of each $\Gamma_i$ as being glued to $M^\Gamma_0$ along $\Gamma_i$, forming the space $\Gamma=M^\Gamma_0\cup_{i=1}^m F(\Gamma_i)\subset L^\infty(M)$. We take a cycle representing the fundamental class of $M^\Gamma_0$. Since the boundary of the fillings $F(\Gamma_i)$ sum to equal the boundary of this cycle, we can combine them to form a new $n$-dimensional cycle. This space can itself by filled by a singular chain. Using Inequality \ref{eqn:fillrad}, we have
\begin{align*}
    \fillrad(\Gamma)
    &\leq \frac{n}{2}\vol_n(\Gamma)^{1/n}\\
    &\leq \frac{n}{2}\left(\vol_n(M_0^\Gamma) + \sum_{i=1}^m \vol_n (F(\Gamma_i))\right)^{1/n}\\
    &\leq \frac{n}{2}
    \left(\vol_n(M)+m\epsilon
    \right)^{1/n}.
\end{align*}
By redefining $\epsilon$, we can thus choose a filling $F(\Gamma)$ of $\Gamma$ that lies within distance $(n/2)(\vol_n(M))^{1/n}+\epsilon$ of $\Gamma$. 
\par 
In summary, we have defined the following spaces:
\begin{itemize}
    \item 
    $M^\Sigma_i$, $i\geq 1$, are locally-convex infinite ends of $M$, with $\partial M^\Sigma_i=\Sigma_i$ and the respective bounded core $M^\Sigma_0$.
    \item
    $M^\Gamma_i$, $i\geq 1$, are infinite ends of $M$ cut off by hypersurfaces $\Gamma_i$ with small volume, and the respective bounded core $M^\Gamma_0$.
    \item 
    $C^\Sigma_i$, $i\geq 1$, is the one-point compactification of $\overline{M_i^\Sigma}$, and $C^\Sigma = M_0^\Sigma\cup_iC^\Sigma_i$.
    \item 
    $F(\Gamma_i)$ is a union of small chains that fill each connected component of $\Gamma_i$.
    \item 
    $\Gamma=M^\Gamma_0\cup_{i=1}^m F(\Gamma_i)$
    \item 
    $F(\Gamma)$ is a small chain filling $\Gamma$.
\end{itemize}
\par 
We are now ready to prove our main lemma describing the impossible extension of $\phi$.
First, we define the inclusion
$$j:(\overline{M^\Gamma_0},\cup_{i=1}^m \Gamma_i)\into (F(\Gamma),\cup_{i=1}^m F(\Gamma_i)).$$ 
Recall that we have also previously defined the quotient map
$$q:(C^\Sigma,\cup_{i=1}^p C_i^\Sigma)\to C^\Sigma/\cup_{i=1}^p C_i^\Sigma.$$
Consider a map $\phi:(\overline{M^\Gamma_0},\cup_{i=1}^m \Gamma_i)\to (C^\Sigma,C^\Sigma_i)$ (not necessarily homotopic to the inclusion map, although that is how we will ultimately apply this lemma). We will prove that if $q\circ \phi$ is continuous and induces a non-zero map on $n$-dimensional homology in $\bZ_2$ co-efficients, then $\phi$ cannot have its domain extended to $(F(\Gamma),\cup_{i=1}^m F(\Gamma_i))$. We only require that $q\circ \phi$ is continuous, not $\phi$ itself, since we will eventually apply this lemma to a map $\phi$ defined using a weak filling of cages (and hence we can only guarantee that $q\circ \phi$ will  be continuous as per the fifth property in Definition \ref{def:filling_cages}).
Figure \ref{fig:map_schematic} provides a diagram of the spaces and maps in question.

\begin{lemma}[c.f. Lemma 2.1 of \cite{clnr}]
    \label{lemma:no_extension}
    Consider a map $$\phi:(\overline{M^\Gamma_0},\cup_{i=1}^m \Gamma_i)\to (C^\Sigma,\cup_{i=1}^p C_i^\Sigma)$$
    such that $q\circ \phi$ is continuous and $(q\circ \phi)_*$ is a non-zero map on $n$-dimensional homology in $\bZ_2$-coefficients.
    Then there is no map $$\psi:(F(\Gamma),\cup_{i=1}^m F(\Gamma_i))\to (C^\Sigma,\cup_{i=1}^p C_i^\Sigma)$$ such that $\phi=\psi\circ j$ and such that $q\circ \psi$ is continuous.
\end{lemma}

\begin{proof}
    Suppose such a map $\psi$ exists.  
    Note that $j_*=0$ in $n$-dimensional homology in $\bZ_2$ coefficients, because $\partial (F(\Gamma))=\Gamma=M_0^\Gamma \cup_{i=1}^m F(\Gamma_i)$.
    Therefore $(q\circ\phi)_*=(q\circ\psi)_*\circ j_*=0$ in dimension $n$. 
    This contradicts our assumption that $(q\circ \phi)_*$ is a non-zero map on $n$-dimensional homology in $\bZ_2$-coefficients. Therefore there can be no such map $\psi$.
\end{proof}

\begin{figure}[ht]
    \centering
    \includegraphics[width=1\textwidth]{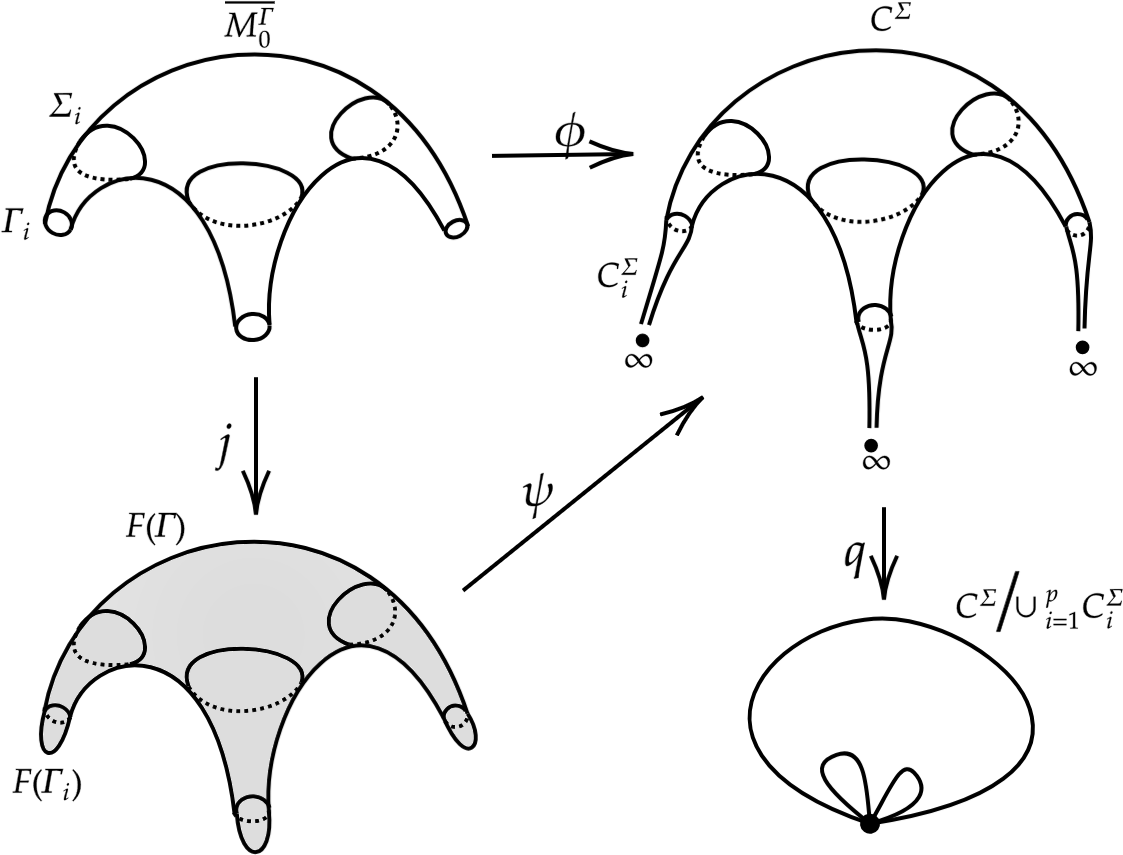}
    \caption{The maps $q,j,\phi$ and $\psi$.}
    \label{fig:map_schematic}
\end{figure}

\subsection{Proof of Main Theorem}
\label{sec:main}

Now we prove our main theorem by building our specific examples of maps $\phi$ and $\psi$ satisfying Lemma \ref{lemma:no_extension} in order to conclude the existence of a short geodesic flower.
\begin{maintheorem}
    Let $M$ be a complete, non-compact $n$-dimensional Riemannian manifold with locally convex ends and finite volume. Then $M$ admits a non-trivial geodesic net with one vertex, at most $(n+2)(n+1)/2$ edges, and total length at most $(n+2)(n+1)(n/2)\vol(M)^{1/n}.$
\end{maintheorem} 
\begin{figure}
    \centering
    \includegraphics[width=0.8\textwidth]{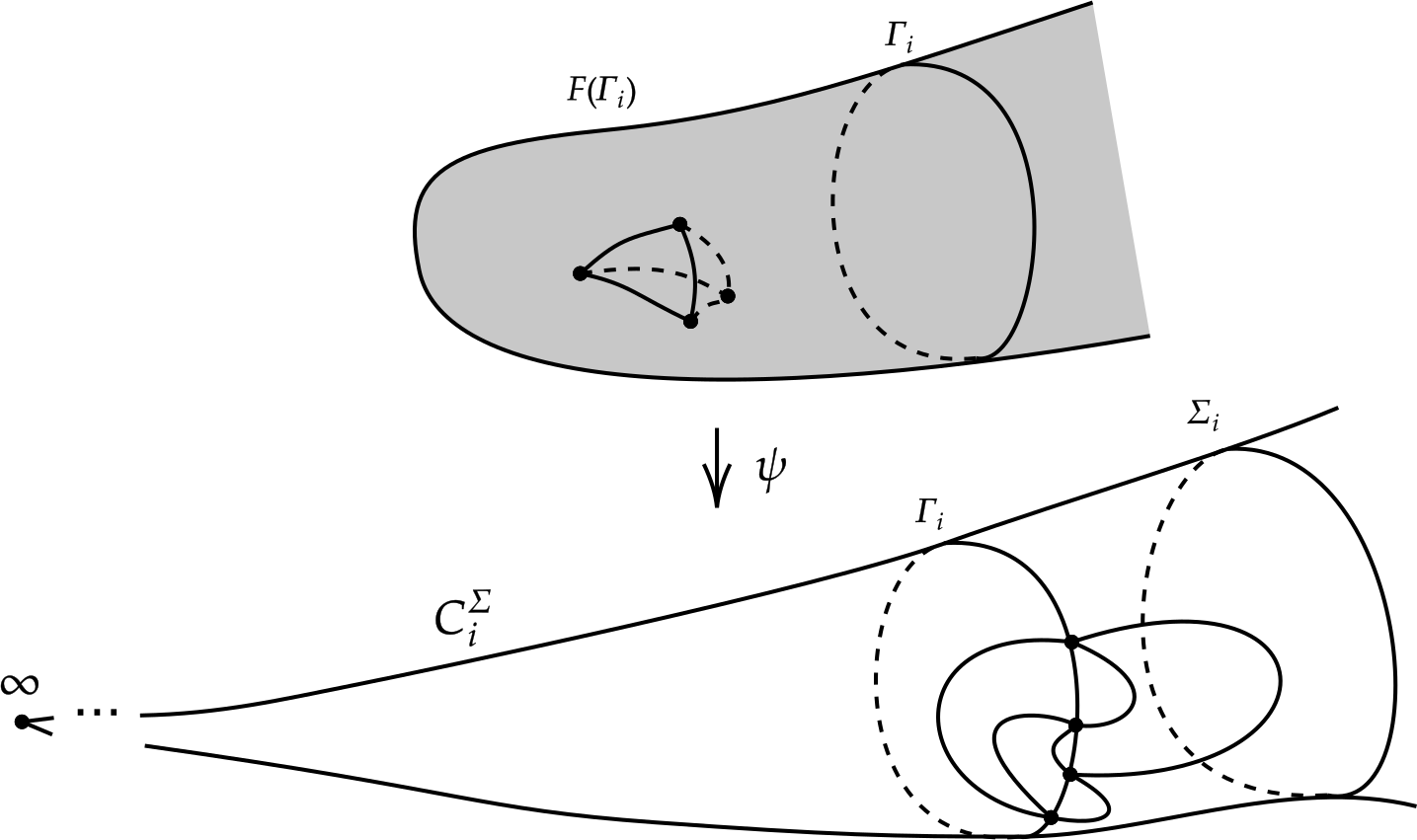}
    \caption{The image of a 3-cage under the map $\psi$ described in the proof of our main theorem.}
    \label{fig:1skeleton}
\end{figure}
\begin{proof} 
    We first show that for all sufficiently small $\delta>0$, there exists a geodesic flower on $M$ with one vertex, at most $(n+2)(n+1)/2$ edges, and total length at most $$L_\delta=\frac{(n+2)(n+1)n}{2}\vol(M)^{1/n}+\delta.$$
    In fact, this flower will lie within distance $2L_\delta$ of $M_0^\Gamma$.
    Fix the desired $\delta$ and define $\delta_1=\delta/((n+2)(n+1))$. 
    We define $\Gamma_i$, $M^\Gamma_i$, and so on, as summarized in Section \ref{sec:extension}. Subdivide the triangulation of $F(\Gamma)$ so that it consists solely of simplices with diameter at most $\delta_1/3$. Moreover, ensure that no simplex has interior intersecting both $M_0^\Sigma$ and $\overline{M_i^\Sigma}$ for any $i>0$. This ensures that if all vertices of a simplex lie in $\overline{M_i^\Sigma}$ for some $i>0$, then the simplex lies entirely in $\overline{M_i^\Sigma}$.
    \par
    We define a map 
    $$\psi:(F(\Gamma),\cup_{i=1}^m F(\Gamma_i))\to (C^\Sigma,\cup_{i=1}^p C_i^\Sigma)$$
    on the simplices of the triangulation of $F(\Gamma)$ as follows. 
    \begin{itemize}
    \item 
    \textbf{0-Skeleton.} Consider a point $x$ in the 0-skeleton of the triangulation.
    If $x\in \overline{M_0^\Gamma}$, map $x$ to itself (viewing $\overline{M_0^\Gamma}$ as a subset of $C^\Sigma$).
    If $x\in F(\Gamma_i)\setminus \Gamma_i$ for some $i$, map $x$ to any choice of nearest point in $\Gamma_i\subset C_i^\Sigma$.    If $x\in F(\Gamma)\setminus \Gamma$, we again map $x$ to any choice of nearest point in $\overline{M^\Gamma_0}\subset C^\Sigma$.
    \item 
    \textbf{1-Skeleton.}
    Consider an edge in the triangulation joining some pair of points $x$ and $y$.
    We map this edge to any choice of minimizing geodesic in $M\subset C^\Sigma$ connecting $\psi(x)$ and $\psi(y)$. 
    \end{itemize}
    We want to check that, as defined so far, $\psi$ maps $\cup_{i=1}^m F(\Gamma_i)$ to $\cup_{i=1}^p C_i^\Sigma$ so that it is a valid map of pairs. This is true by definition for vertices.
    If $x,y\in F(\Gamma_i)$ for some $i$ such that $x$ and $y$ are connected by an edge, then $\psi(x),\psi(y)\in \Gamma_i$ and 
    \begin{align*} 
        d_\infty(\psi(x),\psi(y))
        &\leq 
        d_\infty(\psi(x),x)
        +
        d_\infty(x,y)
        + d_\infty(y,\psi(y))
    \end{align*}        
    Note that if $x\in F(\Gamma_i)$ for some $i$, then
    $$
    d_\infty(x,\psi(x))\leq \fillrad(\Gamma_i),$$
    where we recall that $d_\infty$ is the sup-norm on the ambient space $L^\infty(M)$.
    Recall by Section \ref{sec:extension} that we can choose $\Gamma_i$ to have small enough volume that
    $\fillrad(\Gamma_i)\leq \delta_1/3.$ Therefore
    \begin{align*} 
        d_\infty(\psi(x),\psi(y))&\leq
        \fillrad(\Gamma_i)
        +
        \delta_1/3+\fillrad(\Gamma_i)
        <\delta_1.
    \end{align*}
    Therefore by local convexity of $M_i^\Sigma$, for a sufficiently small choice of $\delta$ (and hence of $\delta_1$) there is a unique minimizing geodesic connecting $\psi(x)$ and $\psi(y)$, and it lies in $\overline{M^\Sigma_i}\subset C^\Sigma_i$. Consequently, any 1-simplex in $F(\Gamma_i)$ is mapped into $\overline{M^\Sigma_i}\subset C_i^\Sigma$ by $\psi$, so that it is indeed a map of pairs.
    \par
    We would like to define $\psi$ on higher-dimensional simplices using a weak filling of cages. We note that the image of every $j$-cage under $\psi$ has length bounded in terms of $\delta$, $n$, and the volume of $M$ by the following argument. An edge in our triangulation connecting two points $x$ and $y$ is mapped to a minimizing geodesic connecting $\psi(x)$ and $\psi(y)$. Therefore, as noted above, every edge of a cage is mapped to a curve of length at most 
    $$
    d_\infty(\psi(x),\psi(y))
    \leq 
    d_\infty(\psi(x),x)
    +
    d_\infty(x,y)
    +
    d_\infty(y,\psi(y)).    
    $$
    By construction, if $z\in M^\Gamma_0$ then $z=\psi(z)$ so $d_\infty(z,\psi(z))=0$. If $z\in F(\Gamma_i)$ then $d_\infty(z,\psi(z))\leq \delta_1/3$ as noted above. Lastly, suppose that $z\in F(\Gamma)\setminus \Gamma$. We would like to bound $d_\infty(z,\psi(z))$ in terms of the filling radius of $\Gamma$, but $\psi(z)$ is a closest point to $z$ in $\overline{M_0^\Gamma}$, not a closest point in $\Gamma$. Let $w$ be the closest point in $\Gamma$ to $z$. If $w$ lies in $\overline{M_0^\Gamma}$, then
    $$d(z,\psi(z))\leq d(z,w)\leq \fillrad(\Gamma)$$
    since in this case $\psi(z)$ is as close to $z$ as any point in $\overline{M_0^\Gamma}$. The other possibility is that $w\not \in\overline{M_0^\Gamma}$, and hence $w\in F(\Gamma_i)\setminus \Gamma_i$ for some $i$.
    By the definition of $\psi$, we have $\psi(w)\in \Gamma_i\subset \overline{M_0^\Gamma}$ (see Figure \ref{fig:psi_w}). Therefore
    \begin{align*}
        d_\infty(z,\psi(z))&\leq d_\infty(z,\psi(w))\\
        &\leq d_\infty(z,w) + d_\infty(w,\psi(w))\\
        &
        \leq 
        \fillrad(\Gamma) + \fillrad(\Gamma_i)
    \end{align*}
    \begin{figure}
        \centering
        \includegraphics[width=0.5\textwidth]{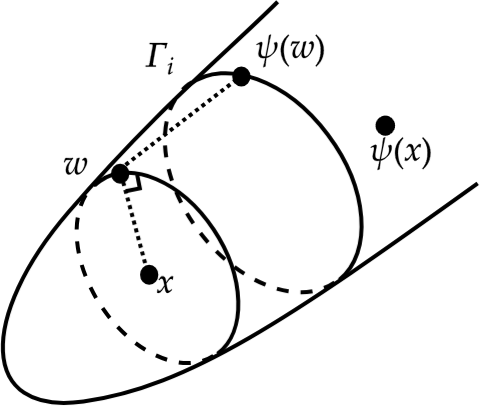}
        \caption{Possible configurations of the points $x,\psi(x),w$ and $\psi(w)$.}
        \label{fig:psi_w}
    \end{figure}
    This gives us an upper bound for $d_\infty(z,\psi(z))$ in all possible cases.
    \par 
    It remains to bound $d_\infty(x,y)$.
    Lastly, if $x$ and $y$ are connected by an edge in our triangulation, then $d_\infty(x,y)\leq\delta_1/3$ because our triangulation is very fine. Thus 
    \begin{align*}
        d_\infty(\psi(x), \psi(y))
        &\leq \delta_1/3+2\max\{\fillrad(\Gamma)+\max_i \fillrad(\Gamma_i),\delta_1/3\}.
    \end{align*}
    Using the inequalities established in Section \ref{sec:extension}, we have
    \begin{align*}
        d_\infty(\psi(x), \psi(y))\leq
        \delta_1+
        n
    \left(\vol_n(M)
    \right)^{1/n}+4\epsilon
    \end{align*}
    Note that this is the distance as measured in $L^\infty(M)$. However, because the embedding of $M$ in $L^\infty(M)$ is isometric, the distance between $\psi(x)$ and $\psi(y)$ as measured with respect to the original metric on $M$ equals $d_\infty(\psi(x),\psi(y))$. Thus an edge in $M$ connecting $\psi(x)$ and $\psi(y)$, being a minimizing geodesic in $M$, has length equal to $d_\infty(\psi(x),\psi(y))$. For suitably small choice of $\epsilon$, we can therefore ensure that the maximum length of the image of any edge, denoted by $L_E$, satisfies
    \begin{align*}
        L_E\leq
        2\delta_1
        +n\vol(M)^{1/n}.
    \end{align*}
    Because each $j$-cage has at most 
    $\binom{j+1}{2}$
    edges and we only need to consider $j\leq n+1$, the image of a $j$-cage under $\psi$ has total length at most
    \begin{align*}
        \binom{j+1}{2}L_E
        &\leq
        \frac{(n+2)(n+1)}{2}
        \left(2\delta_1
        +n\vol(M)^{1/n}\right)\\
        &=
        \delta+\frac{(n+2)(n+1)n}{2}\vol(M)^{1/n}
        =L_\delta.
    \end{align*}
    \par
    We now assume that the $2L_\delta$-neighbourhood of $M_0^\Sigma$ does not admit a geodesic flower of length at most $L_\delta$. By the proof of Theorem \ref{theorem:filling_no_cages}, $M$ admits a weak filling of $\cage^{L_\delta}_n$. We now use this fact to define of $\psi$ on all higher-dimensional simplices.
    \begin{itemize} 
    \item 
    \textbf{$l$-Skeleton, $2\leq l\leq n+1$.} 
    At this point, we have defined $\psi$ on all $k$-cages in the triangulation of $F(\Gamma)$. Consider an $l$-cell $\rho$ of the triangulation of $F(\Gamma)$. Its 1-skeleton is an $l$-cage $\Theta$. By Lemma \ref{lemma:filling}, we can extend $\psi|_\Theta: \Theta\to C^\Sigma$ to a map $g_\Theta:\rho \to C^\Sigma$ such that $q\circ g_\Theta$ is continuous. This defines $\psi$ on $\rho$. 
    \end{itemize}
    Moreover, by the third property of Lemma \ref{lemma:filling}, if $\Theta$ lies in $\overline{M^\Sigma_i}$, then so does its image under $\psi$. In particular, $\psi$ maps any simplex with all vertices in $F(\Gamma_i)$ into $\overline{M^\Sigma_i}\subset C_i^\Sigma$, ensuring it is a valid map of pairs.
    \par    
    This defines the map $\psi$ on all of $F(\Gamma)$.  We define $\phi$ by setting $\phi=\psi\circ j$. Although we have been describing $\phi$ as the inclusion map, it actually differs slightly from the inclusion map on higher-dimensional simplices because we defined it using a weak filling of cages. However, we will prove below that it is homotopic to the inclusion map, which is sufficient for our purposes.
    \par
    We will now show that $\psi$ is an example of the psuedo-extension map from Lemma \ref{lemma:no_extension}. This contradiction implies that no weak filling of $\cage^{L_\delta}_{n+1}$ exists, and hence that a geodesic flower of length at most $L_\delta$ must exist. We need to check that 
    $$q\circ \psi:(F(\Gamma),\cup_{i=1}^m F(\Gamma_i))\to C^\Sigma/\cup_{i=1}^p C_i^\Sigma$$ 
    is continuous, that the map 
    $$\phi=\psi\circ j:(\overline{M^\Gamma_0},\cup_i^m\Gamma_i)\to (C^\Sigma,\cup_{i=1}^p C_i^\Sigma)$$ 
    satisfies the conditions that $q\circ \phi$ is continuous and that $(q\circ\phi)_*$ is not the zero map on $n$-dimensional homology in $\bZ_2$ coefficients.
    \par  
    By Property 5 of Definition \ref{def:filling_cages}, $q\circ \psi$ is continuous. This takes care of the first desired property.
    Next, consider $q\circ\phi=q\circ\psi\circ j$. This map is continuous because $q\circ \psi$ and $j$ are both continuous. We claim that this map is homotopic to the inclusion map $\iota:(\overline{M^\Gamma_0},\cup_i^m\Gamma_i)\into (C^\Sigma,\cup_{i=1}^p C_i^\Sigma)$. Vertices in the triangulation of $\overline{M^\Gamma_0}$ are mapped by $\phi$ to themselves. Edges in $\overline{M^\Gamma_0}$ are mapped by $\phi$ to (unique) minimizing geodesics, which we can assume are equal to the original edges of our triangulation. However, higher dimensional simplices are mapped by shortening the associated cage, and so will not necessarily be mapped to themselves under $\phi$. Since each such simplex is at most $\delta_1/3$ in diameter, it is contained in a metric ball of radius $\delta_1/3$. Therefore if we pick $\delta_1$ smaller than the injectivity radius of $\overline{M^\Gamma_0}$, our shortening process will contract these cages to a point curve within this (convex) metric ball. Therefore the image of these cages under $\phi$ lies within distance $\delta_1/3$ of the original cage. The image of $\overline{M^\Gamma_0}$ under $\phi$ can easily be homotoped to $\overline{M^\Gamma_0}\subset C^\Sigma$, such as by flowing each point along the geodesic connecting it to its preimage. This homotopes $q\circ \phi$ to $\iota$.
    \par 
    This means that $q\circ \phi$ is a non-zero map on $n$-dimensional homology because $\iota$ is a non-zero map by the following argument. 
    Consider a relative cycle in $H_n(\overline{M^\Gamma_0},\cup_i^m\Gamma_i, \bZ_2)$ whose image is $M^\Gamma_0$ modulo $\cup_i^m\Gamma_i$. This cycle represents the fundamental class in relative homology. Because $\overline{M^\Sigma_0}\setminus \overline{M^\Gamma_0}\subset \cup_{i=1}^p C_i^\Sigma$, this relative cycle is mapped by $\iota$ to a relative cycle whose image is $C^\Sigma$ modulo $ \cup_{i=1}^p C_i^\Sigma$. This cycle is non-zero-- in fact, it represents the generator of $H_n(\overline{M^\Sigma_0},\cup_{i=1}^p C_i^\Sigma, \bZ_2)$. Thus $\iota_*$, and hence $(q\circ \phi)_*$, is non-zero.
    \par
    This proves that our maps satisfy the required hypotheses.
    Thus,
    Lemma \ref{lemma:no_extension} states that there is no such map $\psi$. Because the only obstruction to our construction of $\psi$ is the existence of a weak filling of $\cage^{L_\delta}_{n+1}$ (i.e., the hypothesis of Lemma \ref{lemma:filling}), such a filling cannot exist. Therefore by Theorem \ref{theorem:filling_no_cages}, $M$ must admit a non-trivial geodesic flower with at most ${n+1\choose 2}$ edges and of total length at most $L_\delta$.
    \par In order to prove our main theorem, it remains to remove the dependence on $\delta$ in the above length bound.
    Consider the sequence of $\delta$ values $\delta_k=1/k$. We have established that for all sufficiently large $k>0$, $M$ admits a geodesic flower $\Theta_k$ of length at most $$L_\delta=\frac{(n+2)(n+1)n}{2}\vol(M)^{1/n}+1/k.$$ We would like to show that a subsequence of these geodesic flowers converges to a geodesic flower of length at most $$L_\infty=\frac{(n+2)(n+1)n}{2}\vol(M)^{1/n}.$$  
    However, we must avoid the possibility that these flowers escape to infinity. In fact, the construction of a weak filling of $\cage^L_{n+1}$ is obstructed by the existence of a short geodesic flower of length at most $L$ that lies within a distance $2L$ of $\overline{M_0^\Sigma}$, as remarked at the end of Section \ref{sec:nets}. Therefore, since $L_k\leq L_1$, our sequence of nets lies in a compact set. Therefore by the Arzel\`a--Ascoli theorem, a subsequence of our geodesic flowers converges to a geodesic flower of length at most $L_\infty$, proving our theorem.
\end{proof}

\bibliography{biblio}

\end{document}